\documentclass[reqno]{amsart}

\usepackage[T1]{fontenc}
\usepackage[utf8]{inputenc}
\usepackage{amsmath,amssymb,amsthm}
\usepackage{microtype}
\usepackage[hidelinks]{hyperref}

\hypersetup{
  pdftitle={A direct proof of the locally dense graphon inequality},
  pdfauthor={Dean Menezes},
  pdfsubject={2020 Mathematics Subject Classification: Primary 05C35; Secondary 05C80, 28A20},
  pdfkeywords={graphons, locally dense graphs, copositive kernels, randomized rounding, atomless probability spaces}
}

\theoremstyle{plain}
\newtheorem{theorem}{Theorem}
\newtheorem{proposition}[theorem]{Proposition}
\newtheorem{lemma}[theorem]{Lemma}
\newtheorem{corollary}[theorem]{Corollary}
\theoremstyle{remark}
\newtheorem{remark}[theorem]{Remark}

\newcommand{\one}{\mathbf 1}
\newcommand{\E}{\mathbb E}
\newcommand{\N}{\mathbb N}
\newcommand{\dd}{\,\mathrm d}
\newcommand{\norminf}[1]{\lVert #1\rVert_\infty}
\newcommand{\normone}[1]{\lVert #1\rVert_1}

\begin{document}

\title[A direct proof of the locally dense graphon inequality]
      {A direct proof of the locally dense graphon inequality}

\author{Dean Menezes}
\address{The University of Texas at Austin, Austin, Texas 78712, USA}
\email{dean.menezes@utexas.edu}
\date{August 13, 2026}

\begin{abstract}
Brada\v{c}, Sudakov, and Wigderson characterized $p$-locally dense graphons by
a quadratic inequality for all bounded nonnegative functions. Their proof uses
Reiher's finite lemma and graphon approximation, and they asked for a direct
proof. We give one by rounding simple functions. Divide each level set into
$m$ equal-measure pieces and retain each piece independently with probability
equal to its level. Off-diagonal terms agree in expectation; the diagonal
error is at most $\norminf{W-p}/(4m)$. Letting $m\to\infty$ and then
approximating in $L^1$ proves the inequality. A three-atom example shows that
atomlessness is necessary if arbitrary probability spaces are allowed.
\end{abstract}

\maketitle

\section{Introduction}

Let $(\Omega,\Sigma,\mu)$ be a probability space. A symmetric measurable
kernel $W\colon\Omega^2\to[0,1]$ is \emph{$p$-locally dense} if
\begin{equation}\label{eq:local-density}
  \iint_{U\times U}W(x,y)\dd\mu(x)\dd\mu(y)
  \geq p\,\mu(U)^2
  \qquad(U\in\Sigma).
\end{equation}
This is the graphon analogue of local density in finite graphs; see
\cite{Lovasz}. Brada\v{c}, Sudakov, and Wigderson
\cite[Lemma~2.8]{BSW} proved that, on the usual atomless graphon space,
\eqref{eq:local-density} is equivalent to
\begin{equation}\label{eq:fractional}
  \iint_{\Omega^2}f(x)W(x,y)f(y)\dd\mu(x)\dd\mu(y)
  \geq p\,\normone{f}^{2}
\end{equation}
for every bounded measurable $f\colon\Omega\to[0,\infty)$. Taking
$f=\one_U$ proves one implication. Their proof of the converse passes through
a finite lemma of Reiher \cite{Reiher} and graphon approximation. They asked
for a proof that works directly with $W$.

\begin{theorem}\label{thm:main}
Let $(\Omega,\Sigma,\mu)$ be an atomless probability space, let $p\in[0,1]$,
and let $W\colon\Omega^2\to[0,1]$ be symmetric and measurable. If
\eqref{eq:local-density} holds, then \eqref{eq:fractional} holds for every
bounded measurable $f\colon\Omega\to[0,\infty)$.
\end{theorem}

The proof is an averaging argument. Write
$f=\sum_i a_i\one_{A_i}$ and divide each $A_i$ into $m$ equal parts. Retain
each part independently with probability $a_i$. Off-diagonal interactions
then have the correct expectation, while the total diagonal error is at most
$\norminf{W-p}/(4m)$. Local density controls every rounded indicator; letting
$m\to\infty$ proves the simple case, and an $L^1$ estimate proves the rest.

\section{Random rounding}

For a bounded symmetric measurable kernel $K$ on $\Omega^2$, put
\begin{equation}\label{eq:q-def}
  q(f):=\iint_{\Omega^2}f(x)K(x,y)f(y)\dd\mu(x)\dd\mu(y).
\end{equation}
For measurable $B,C\subseteq\Omega$, abbreviate
\begin{equation}\label{eq:KBC}
  K(B,C):=\iint_{B\times C}K(x,y)\dd\mu(x)\dd\mu(y).
\end{equation}
Thus, for disjoint measurable sets $B_i$,
\begin{equation}\label{eq:block-form}
  q\!\left(\sum_i c_i\one_{B_i}\right)
  =\sum_{i,j}c_i c_j K(B_i,B_j).
\end{equation}

Set $K=W-p$. Then
\begin{equation}\label{eq:K-bound}
  \norminf{K}\leq\max\{p,1-p\}\leq1
\end{equation}
and
\begin{equation}\label{eq:reduction}
  q(f)=\iint f(x)W(x,y)f(y)\dd\mu(x)\dd\mu(y)
       -p\,\normone{f}^{2}.
\end{equation}
Consequently, \eqref{eq:local-density} says that $q(\one_U)\geq0$ for every
$U\in\Sigma$. By homogeneity, Theorem~\ref{thm:main} reduces to the next
statement.

\begin{theorem}[Kernel lemma]\label{thm:kernel}
Let $(\Omega,\Sigma,\mu)$ be atomless, and let
$K\in L^\infty(\Omega^2)$ be symmetric. If
\[
  q(\one_U)\geq0\qquad(U\in\Sigma),
\]
then $q(f)\geq0$ for every measurable $f\colon\Omega\to[0,1]$.
\end{theorem}

We need the intermediate-value property of atomless measures. The short proof
is included; see also \cite[\S215]{Fremlin}.

\begin{lemma}[Equal-measure splitting]\label{lem:splitting}
If $(\Omega,\Sigma,\mu)$ is atomless, then every $A\in\Sigma$ and every
$m\in\N$ have a measurable partition
\[
  A=A_1\sqcup\cdots\sqcup A_m,
  \qquad
  \mu(A_s)=\frac{\mu(A)}m\quad(1\leq s\leq m).
\]
\end{lemma}

\begin{proof}
We first show that, for $0\leq t\leq\mu(A)$, some measurable $B\subseteq A$
has measure $t$. Only the case $0<t<\mu(A)$ requires proof.

Every positive-measure set $D$ contains subsets of arbitrarily small positive
measure. Since $D$ is not an atom, it can be divided into two sets of positive
measure; the smaller has measure at most $\mu(D)/2$. Iteration proves the
claim.

Set $B_0=\varnothing$. Given $B_n\subseteq A$ with $\mu(B_n)\leq t$, let
\[
  \varepsilon_n=
  \sup\{\mu(C):C\subseteq A\setminus B_n,\ 
                    \mu(B_n)+\mu(C)\leq t\}.
\]
Choose an admissible $C_n$ with $\mu(C_n)\geq\varepsilon_n/2$, taking
$C_n=\varnothing$ when $\varepsilon_n=0$, and put
$B_{n+1}=B_n\sqcup C_n$. The sets $B_n$ increase, so
$B=\bigcup_nB_n$ has measure at most $t$.

If $\mu(B)<t$, choose
$C\subseteq A\setminus B$ with $0<\mu(C)\leq t-\mu(B)$. This $C$ is
admissible in the definition of every $\varepsilon_n$. Hence
$\mu(C_n)\geq\mu(C)/2$ for all $n$, contrary to
$\sum_n\mu(C_n)\leq\mu(A)$. Thus $\mu(B)=t$.

Now remove from $A$ successive sets of measure $\mu(A)/m$. The last remainder
has the same measure.
\end{proof}

Let
\begin{equation}\label{eq:simple-f}
  f=\sum_{i=1}^n a_i\one_{A_i},
  \qquad 0\leq a_i\leq1,
\end{equation}
where the $A_i$ are pairwise disjoint. Fix $m\in\N$ and use
Lemma~\ref{lem:splitting} to write
\begin{equation}\label{eq:block-split}
  A_i=\bigsqcup_{s=1}^mA_{i,s},
  \qquad
  \mu(A_{i,s})=\frac{\mu(A_i)}m.
\end{equation}
Choose independent Bernoulli variables $\xi_{i,s}$ with
$\mathbb P(\xi_{i,s}=1)=a_i$, and set
\begin{equation}\label{eq:random-U}
  U=\bigcup_{\xi_{i,s}=1}A_{i,s}.
\end{equation}
The set $U$ takes finitely many values, so the expectations below are finite
sums.

\begin{proposition}[Rounding identity]\label{prop:rounding}
With the notation above,
\begin{equation}\label{eq:rounding-identity}
  \E q(\one_U)=q(f)+\Delta_m,
  \qquad
  \Delta_m=\sum_{i=1}^n a_i(1-a_i)
      \sum_{s=1}^mK(A_{i,s},A_{i,s}),
\end{equation}
and
\begin{equation}\label{eq:error-bound}
  |\Delta_m|
  \leq \frac{\norminf{K}}m
         \sum_{i=1}^n a_i(1-a_i)\mu(A_i)^2
  \leq \frac{\norminf{K}}{4m}
         \sum_{i=1}^n\mu(A_i)^2
  \leq \frac{\norminf{K}}{4m}.
\end{equation}
\end{proposition}

\begin{proof}
The pieces in \eqref{eq:block-split} are disjoint, and therefore
\[
  \one_U=\sum_{i=1}^n\sum_{s=1}^m\xi_{i,s}\one_{A_{i,s}}.
\]
Equation \eqref{eq:block-form} gives
\[
  q(\one_U)=
  \sum_{(i,s),(j,t)}
  \xi_{i,s}\xi_{j,t}K(A_{i,s},A_{j,t}).
\]
For $(i,s)\neq(j,t)$, independence gives
$\E(\xi_{i,s}\xi_{j,t})=a_i a_j$. On the diagonal,
$\E(\xi_{i,s}^2)=a_i=a_i^2+a_i(1-a_i)$. Taking expectations and then using
finite additivity in both arguments yields
\[
\begin{split}
  \E q(\one_U)
  &=\sum_{(i,s),(j,t)}a_i a_jK(A_{i,s},A_{j,t})+\Delta_m\\
  &=\sum_{i,j}a_i a_jK(A_i,A_j)+\Delta_m
   =q(f)+\Delta_m.
\end{split}
\]
This proves \eqref{eq:rounding-identity}.

For each $i$ and $s$,
\[
  |K(A_{i,s},A_{i,s})|
  \leq\norminf{K}\mu(A_{i,s})^2
  =\frac{\norminf{K}}{m^2}\mu(A_i)^2.
\]
Summing over $s$ proves the first inequality in \eqref{eq:error-bound}. The
second uses $a_i(1-a_i)\leq1/4$, and the third uses
\[
  \sum_i\mu(A_i)^2
  \leq\left(\sum_i\mu(A_i)\right)^2
  \leq1.
\]
\end{proof}

\begin{corollary}\label{cor:simple}
Under the hypotheses of Theorem~\ref{thm:kernel}, $q(f)\geq0$ for every simple
measurable $f\colon\Omega\to[0,1]$.
\end{corollary}

\begin{proof}
Every value of $U$ is measurable, so $q(\one_U)\geq0$ for every outcome.
Proposition~\ref{prop:rounding} gives
\[
  q(f)=\E q(\one_U)-\Delta_m
  \geq-|\Delta_m|
  \geq-\frac{\norminf{K}}{4m}.
\]
This holds for every $m\in\N$; hence $q(f)\geq0$.
\end{proof}

\section{Passage to general functions}

\begin{lemma}[$L^1$ estimate]\label{lem:continuity}
If $f,g\colon\Omega\to[0,1]$ are measurable, then
\begin{equation}\label{eq:L1-bound}
  |q(f)-q(g)|
  \leq\norminf{K}\bigl(\normone{f}+\normone{g}\bigr)\normone{f-g}
  \leq2\norminf{K}\normone{f-g}.
\end{equation}
\end{lemma}

\begin{proof}
For
\[
  b(h,k)=\iint h(x)K(x,y)k(y)\dd\mu(x)\dd\mu(y),
\]
we have
\[
  q(f)-q(g)=b(f-g,f)+b(g,f-g).
\]
Since
\[
  |b(h,k)|\leq\norminf{K}\normone{h}\normone{k},
\]
the first inequality in \eqref{eq:L1-bound} follows. The second follows from
$\normone{f},\normone{g}\leq1$.
\end{proof}

\begin{proof}[Proof of Theorem~\ref{thm:kernel}]
Given measurable $f\colon\Omega\to[0,1]$, define the simple functions
\[
  f_N=2^{-N}\lfloor2^Nf\rfloor\qquad(N\in\N).
\]
Then $0\leq f_N\leq1$ and $\normone{f-f_N}\leq2^{-N}$. By
Corollary~\ref{cor:simple} and Lemma~\ref{lem:continuity},
\[
  q(f)\geq q(f_N)-2\norminf{K}\normone{f-f_N}
       \geq-2^{1-N}\norminf{K}.
\]
Let $N$ tend to infinity.
\end{proof}

\begin{proof}[Proof of Theorem~\ref{thm:main}]
Set $K=W-p$. If $\norminf{f}=0$, then \eqref{eq:fractional} is immediate.
Otherwise apply Theorem~\ref{thm:kernel} to $f/\norminf{f}$ and multiply by
$\norminf{f}^{2}$.
\end{proof}

The argument gives a finite-scale form of the theorem.

\begin{corollary}[Finite-scale estimate]\label{cor:quantitative}
Put $M=\max\{p,1-p\}$. If
$f=\sum_i a_i\one_{A_i}\colon\Omega\to[0,1]$ is simple and the $A_i$ are
disjoint, then for every $m\in\N$,
\begin{equation}\label{eq:quant-simple}
  \iint f(x)W(x,y)f(y)\dd\mu(x)\dd\mu(y)
  \geq p\normone{f}^{2}
      -\frac{M}{4m}\sum_i\mu(A_i)^2.
\end{equation}
For measurable $f\colon\Omega\to[0,1]$, put
$f_N=2^{-N}\lfloor2^Nf\rfloor$. Then for all $m,N\in\N$,
\begin{equation}\label{eq:quant-general}
  \iint f(x)W(x,y)f(y)\dd\mu(x)\dd\mu(y)
  \geq p\normone{f}^{2}
      -M\left(\frac1{4m}+2^{1-N}\right).
\end{equation}
\end{corollary}

\begin{proof}
For simple $f$, combine Proposition~\ref{prop:rounding} with
$\E q(\one_U)\geq0$ and \eqref{eq:K-bound}. This gives
\eqref{eq:quant-simple}. Apply the same estimate to $f_N$, use
$\sum_i\mu(A_i)^2\leq1$, and then apply Lemma~\ref{lem:continuity} with
$\normone{f-f_N}\leq2^{-N}$ to obtain \eqref{eq:quant-general}.
\end{proof}

\section{The atomic obstruction}\label{sec:atoms}

The kernel lemma fails without atomlessness, even on three points.

\begin{proposition}\label{prop:atomic-example}
Let $\Omega=\{1,2,3\}$ carry the uniform probability measure, let
$p=\tfrac12$, and let
\[
  W=
  \begin{pmatrix}
    1&0&0\\
    0&1&\tfrac34\\
    0&\tfrac34&1
  \end{pmatrix}.
\]
Then \eqref{eq:local-density} holds for every $U\subseteq\Omega$, but
\eqref{eq:fractional} fails for $f=(1,\tfrac23,\tfrac23)$.
\end{proposition}

\begin{proof}
Let $J$ be the all-ones matrix and put
\[
  M=2W-J=
  \begin{pmatrix}
    1&-1&-1\\
    -1&1&\tfrac12\\
    -1&\tfrac12&1
  \end{pmatrix}.
\]
For $K=W-1/2$, uniform measure gives
$18q(h)=h^{\mathsf T}Mh$. The values of
$\one_U^{\mathsf T}M\one_U$ are $0$ for
$U\in\{\varnothing,\{1,2\},\{1,3\},\Omega\}$, $1$ for the singletons, and
$3$ for $U=\{2,3\}$. Hence $q(\one_U)\geq0$ for every $U$. For
$f=(1,\tfrac23,\tfrac23)$, however, $f^{\mathsf T}Mf=-1/3$ and
$q(f)=-1/54$.
\end{proof}

\begin{remark}\label{rem:convexity}
Because $M_{ii}=1$, the function $a\mapsto a^{\mathsf T}Ma$ is convex in each
coordinate. All eight vertex values on $[0,1]^3$ are nonnegative, but the
value at $(1,\tfrac23,\tfrac23)$ is $-1/3$. Thus separate convexity does not
force a minimum to occur at a vertex.
\end{remark}

\begin{remark}\label{rem:failure}
In the singleton decomposition
$f=\one_{\{1\}}+\tfrac23\one_{\{2\}}+\tfrac23\one_{\{3\}}$, no block can be
divided into $m>1$ equal-measure pieces. At $m=1$ the correction in
\eqref{eq:rounding-identity} is $\Delta_1=2/81$, and
$\E q(\one_U)=q(f)+\Delta_1=1/162\geq0$ although $q(f)=-1/54$.
Atomlessness permits arbitrarily fine splitting and thereby forces
$\Delta_m\to0$.
\end{remark}

\end{document}